\documentclass[a4paper, 11pt]{amsart}
\usepackage{amsfonts, amsmath, amssymb, amsthm,mathtools, url,dsfont}
\usepackage{graphicx}
\usepackage[english] {babel}
\usepackage{enumitem}
\usepackage{array}
\usepackage{csquotes}
\MakeOuterQuote{"}

\newtheorem{thm}{Theorem}
\newtheorem{cor}[thm]{Corollary}
\newtheorem{conj}[thm]{Conjecture}
\newtheorem{lemma}[thm]{Lemma}

\newtheorem*{remark}{Remark}

\newcommand{\calE}{\mathcal E^*}
\newcommand{\calA}{\mathcal A}
\newcommand{\ve}{\varepsilon}
\newcommand {\bO}{\mathrm{O}}
\newcommand {\lo}{\mathrm{o}}
\newcommand {\e}{\mathrm{e}}
\newcommand {\dd}{\mathrm{d}}

\begin{document}
\markboth{C. Aistleitner and L. Kaziulyt\.{e}}
{xxxxyyy}

%%%%%%%%%%%%%%%%%%%%%%%%%%%%%%%%%%%%%%%%%%%%%%%
\title[Variants of Khintchine's theorem]{Variants of Khintchine's theorem in metric Diophantine approximation}	

\author{Laima Kaziulyt\.{e}   }
%	\email{aistleitner@math.tugraz.at\\
	%	laima.kaziulyte@gmail.com} }
	
	\maketitle
		
\begin{abstract}
New results towards the Duffin-Schaeffer conjecture, which is a fundamental unsolved problem in metric number theory, have been established recently assuming extra divergence. Given a non-negative function $\psi: \mathbb{N}\to\mathbb{R}$ we denote by $W(\psi)$ the set of all $x\in\mathbb{R}$ such that $|nx-a|<\psi(n)$ for infinitely many $a,n$. Analogously, denote $W'(\psi)$ if we additionally require $a,n$ to be coprime. Aistleitner et al. \cite{chris} proved that $W'(\psi)$ is of full Lebesgue measure if there exist an $\ve>0$ such that $\sum_{n=2}^\infty\psi(n)\varphi(n)/(n(\log n)^\ve)=\infty$. This result seems to be the best one can expect from the method used. Assuming the extra divergence $\sum_{n=2}^\infty\psi(n)/(\log n)^\ve=\infty$ we prove that $W(\psi)$ is of full measure. This could also be deduced from the result in \cite{chris}, but we believe that our proof is of independent interest, since its method is totally different from the one in \cite{chris}. As a further application of our method, we prove that a variant of Khintchine's theorem is true without monotonicity, subject to an additional condition on the set of divisors of the support of $\psi$.

\end{abstract}



%%%%%%%%%%%%%%%%%%%%%%%%%%%
\section{Introduction}

To be concise, throughout the paper we use the following notation: $p$ denotes a prime number, $\gamma$ is the Euler-Mascheroni constant, $c$ denotes a positive constant which might be different in different relations, $(m,n)$ stands for the greatest common divisor of two integers $m,n$. We write $\lambda$ for Lebesgue measure and $\varphi, \mu, \zeta$ for the Euler's totient, M\"obius and Riemann zeta functions, respectively. We shall also use the standard Landau's symbols $\bO$ and $ \lo$ to compare the order of magnitude of functions in the neighbourhood of infinity as well as Vinogradov's notation $\ll$. 

Let $\psi: \mathbb{N}\to[0,1/2]$ be any function. For each positive integer $n$ define $\mathcal{E}_n\subseteq\mathbb{R}/\mathbb{Z}$ by
\begin{align*}
\mathcal{E}_n := \bigcup_{a=1}^n \left( \frac{a - \psi(n)}{n},\frac{a+\psi(n)}{n} \right).
\end{align*}
Further, define $\mathcal{E}'_n\subseteq\mathbb{R}/\mathbb{Z}$ similarly, but with union restricted to those $a$, coprime to $n$. Write $W(\psi), W'(\psi)$ for the set of $x\in[0,1)$ which are contained in infinitely many sets 
$\mathcal{E}_n, \mathcal{E}'_n $ respectively.

In 1924 Khintchine \cite{khitc} proved that when $n\psi(n)$ is non-increasing, $\lambda(W(\psi))=1$ if
\begin{align}
\sum_{n=1}^\infty\psi(n) \label{sum_of_psi}
\end{align}
diverges. If we do not assume monotonicity, one can construct a function $\psi(n)$ for which (\ref{sum_of_psi}) diverges, yet $\lambda(W(\psi))=0$, as was shown by Duffin and Schaeffer \cite{daffin-schaeffer} in 1941.
In the same paper they raised the following assertion, which became the most important problem in metric number theory, still open to this day. It is an attempt to remove monotonicity from Khinchine's theorem by reducing the sets $\mathcal{E}_n$ to $\mathcal{E}'_n$.

% In the same paper they also showed that it is sufficient to assume weaker monotonicity, namely that $n^{c_1}\psi(n)$ is non-increasing for any real constant $c_1$. Notice that $\lambda(W(\psi))=1$ also implies $\lambda(W'(\psi))=1$, due to the monotonicity of $n\psi(n)$. This is not importantly the case assuming weaker monotonicity $n^{c_1}\psi(n)$ for $c_1<-1$. Thus the distinction arises whether we approximate $x$ by any fractions $a/n$ or by reduced ones. Apparently the latter case has been considerably more investigated in the literature. In neither case, however, necessary and sufficient conditions on the non-negative function $\psi : \mathbb{N} \rightarrow \mathbb{R}$ have been found. The following assertion raised by Duffin and Schaeffer \cite{daffin-schaeffer} became the most important problem in metric number theory, still open to this day.

\begin{conj}
		\text{(Duffin-Schaeffer, 1941)}. 
	We have $\lambda(W'(\psi))=1$ if and only if 
	\begin{align}
	\sum_{n=1}^\infty\psi(n)\frac{\varphi(n)}{n}=\infty. \label{D-S divergence }
	\end{align}
\end{conj}
 Notice that the necessity follows immediately from the first Borel-Cantelli lemma, since $\lambda\left(\mathcal{E}_n\right)=2\psi(n)\varphi(n)/n$.
Various attempts have been made to prove the conjecture posing arithmetic conditions on the function $\psi$ (see for instance \cite{har_p, poll-vaugh, vaaler}). A good exposition of partial results obtained up to the end of the 20th century can be found in the classical G. Harman book \cite{har_b}.  

Recently, however, a new approach has been taken. Haynes, Pollington and Velani in their 2011 paper \cite{div1} proved a weakened conjecture assuming extra divergence. To be precise, they showed that $\lambda(W'(\psi))=1$ supposing
\begin{align*}
\sum_{n=3}^\infty\frac{\psi(n)\varphi(n)}{n\exp(c\log n/ \log\log n)}=\infty. 
\end{align*}
Pursuing their ideas, improvements followed. In \cite{div2} Beresnevich, Harman, Haynes and Velani reduced the extra divergence condition to
\begin{align*}
\sum_{n=16}^\infty\frac{\psi(n)\varphi(n)}{n(\log n)^{\ve\log\log\log n}}=\infty,
\end{align*}
where $\ve>0$ is arbitrary. This has been recently reduced even further to
\begin{align}
\label{weakest_divergence}
\sum_{n=2}^\infty\frac{\psi(n)\varphi(n)}{n(\log n)^{\ve}}=\infty 
\end{align}
 by Aistleitner et al. \cite{chris}, where again $\ve>0$ is arbitrary. The main ideas behind these proofs are as follows. Using some kind of extra divergence, the relation 
\begin{align}
\label{overlap_est}
\sum_{1\le m,n\le N}\lambda\left(\tilde{\mathcal{E}_m}\cap\tilde{\mathcal{E}_n}\right)\ll\left(\sum_{1\le n\le N}\lambda\left(\tilde{\mathcal{E}_n}\right)\right)^2
\end{align}
is proved, where $\tilde{\mathcal{E}_n}$ are specific auxiliary sets. Then the application of Lemma~\ref{limsup} below, together with the Gallagher zero-one law \cite{gall}, leads to the results outlined. To get the overlap estimate (\ref{overlap_est}) Pollington and Vaughan's result \cite{poll-vaugh}, that
 \begin{align}
 \label{pol-vaugh}
 \lambda\left(\mathcal{E}_n'\cap\mathcal{E}_m'\right)\ll\lambda\left(\mathcal{E}_n'\right)\lambda\left(\mathcal{E}_m'\right)P(m,n)
 \end{align}
 is used. Here $P(m,n)$ is a factor which depends on the function $\psi$ and on the divisors of $m$ and $n$ in a complicated way. It is known that it 
 % \begin{align*}
  %P(m,n)=\prod_{p|\frac{mn}{(m,n)^2}\atop  p>D(m,n)}\left(1-\frac{1}{p}\right)^{-1}, 
  %\end{align*}
   can be unbounded for some parameters $m,n$ and thus can not be ignored. The result of \cite{div2} was achieved by averaging $P(m,n)$ over the downscaled versions of the sets $\mathcal{E}_m'$ and $\mathcal{E}_n'$, whereas in \cite{chris} this averaging was taken inside the proof of the formula (\ref{pol-vaugh}). The method developed in the three papers \cite{div1, div2, chris} seems to have achieved its best precision in \cite{chris} and whether an even weaker divergence condition than (\ref{weakest_divergence}) is sufficient is not known. Notice that results on the Duffin-Schaeffer conjecture also hold for Khintchine's problem, since $W'(\psi)\subset W(\psi)$.
  
  In our proof we do not make use of Pollington-Vaughan's sieve estimate for the factor $P(m,n)$ in (\ref{pol-vaugh}). In contrast to the method in \cite{chris} and \cite{div2}, our method does not use any averaging argument. The key novel idea of us is to define new partially reduced sets $\mathcal{E}_n^D$ (see next section) to satisfy (\ref{overlap_est}), which are bigger than $\mathcal{E}_n'$, but smaller than $\mathcal{E}_n$. The benefit of moving from $(\mathcal{E}_n)_{n \geq 1}$ to the reduced set system $(\mathcal{E}_n^D)_{n \geq 1}$ is twofold. Firstly, the measure of these two types of sets is the same up to a multiplicative constant. This lets to avoid much of the trouble that appears in the Duffin-Schaeffer conjecture where too much measure is lost. Secondly, the reduced set system $(\mathcal{E}_n^D)_{n \geq 1}$ shows better independence properties than the full set system $(\mathcal{E}_n)_{n \geq 1}$ in Khintchine's setting. By construction, the intersection of two sets $\mathcal{E}_n^D$ and $\mathcal{E}_m^D$ can only be too large if $(m,n)$ is very large. In other words, for given $n$, there is only a small number of indices $m$ for which the overlap of $\mathcal{E}_n^D$ and $\mathcal{E}_m^D$ is too large. The extra divergence assumption allows us to show that the impact of these "problematic" pairs of indices $m$ and $n$ is negligible. We believe that the same reduction method can be successfully applied in other variants of Khintchine's theorem, and we plan to come back to this topic in a future paper.

\section{Statement of Results}

 For $n \geq 1$, let $D = D(n)$ be a positive real number. For every $n \geq 1$ we define a set $\mathcal{E}_n^D \subseteq \mathbb{R} / \mathbb{Z}$ by 
\begin{equation} \label{edef}
\mathcal{E}_n^D := {\bigcup}_{a\in S} \left( \frac{a - \psi(n)}{n},\frac{a+\psi(n)}{n} \right),
\end{equation}
where 
\begin{align}
S=\left\{a\in\{1,\dots, n\} | (a,n)\le D\right\}. 
\label{set S}
\end{align}
 Note that independently of the choice of $D$ we always have $0 \leq \lambda(\mathcal{E}_n^D) \leq 2 \psi(n)$.

\begin{lemma}
\label{S>=neps}
Fix $\varepsilon\in(0,1)$. Let
$$
D=(\log n)^\ve
$$
for all $n>1$. Then we have
$$
|S|\ge n\ve/10
$$
for all sufficiently large $n$, where $|S|$ denotes the cardinality of the set $S$, defined by~(\ref{set S}).
\end{lemma}

\begin{remark}
	The constant $1/10$ in Lemma \ref{S>=neps} is by no means sharp. We just require $\left|S\right|\gg n\ve$.
	
	The number $D$ in Lemma \ref{S>=neps} is optimal, except for the particular value of $\ve$. More precisely, if we require the set $S$ to contain at least $n\ve$ elements, then $D$ has to be at least $c\left(\log n\right)^{\ve\e^\gamma}$ for some constant $c>0$. This can be seen by letting 
	$$
	n=\prod_{i=1}^kp_i,
	$$
	where $p_i$ are all the consecutive primes up to some sufficiently large $p_k$. Then
	$$
	|S|=\sum_{d|n,\  d\le D}\varphi\left(\frac{n}{d}\right)=\varphi(n)\sum_{l\le D}\frac{\mu(l)^2}{\varphi(l)}.
	$$ 
	 The latter sum is $\log D+\bO(1)$ (see \cite{ward}). Using this and the estimate $\varphi(n)=n\e^{-\gamma}/\log\log n+\bO(n/(\log\log n)^2)$ (see \cite[Theorem 2.9]{mont-vaugh}), from the equality above and our requirement we get $D\ge c(\log n)^{\ve\e^\gamma}$.
\end{remark}	
\begin{cor}
	\label{corollary}
Fix $\varepsilon\in(0,1)$. Let
$$
D=(\log n)^\ve
$$
for all $n>1$. Then we have
$$
\lambda(\mathcal{E}_n^D) \geq \frac{1}{5}\varepsilon \psi(n),
$$
for all sufficiently large $n$.	
\end{cor}	
\begin{proof}
Clearly, $\lambda(\mathcal{E}_n^D) = 2 \psi(n)/n\cdot|S|$ (remember that the sets $\mathcal{E}_n^D$ are defined by (\ref{edef})). The result now follows by inserting the lower bound of $|S|$ from Lemma~\ref{S>=neps}.	
\end{proof}
Corollary says that when we choose $D=(\log n)^\ve$, the measure of the reduced sets $\mathcal{E}_n^D$ is the same as the measure of the original sets $\mathcal{E}_n$, up to a multiplicative constant. This will play a key role in the proof of Theorem \ref{khintchine-divergence}. Note that this is different from the case of the Duffin-Schaeffer conjecture, where for the measure of the reduced sets $\mathcal{E}'_n$ we only have the lower bound $\lambda(\mathcal{E}_n')\gg\lambda(\mathcal{E}_n)/\log\log n$. 

\begin{thm}
	\label{khintchine-divergence}
	We have $\lambda\left(W(\psi)\right)=1$ for any function $\psi: \mathbb{N}\to[0,1/2]$, for which there is a constant $\ve>0$, such that 
	\begin{align}
	\label{divergence}
	\sum_{n=2}^\infty\frac{\psi(n)}{(\log n)^\ve}=\infty.
	\end{align}
\end{thm}	
\begin{remark}
	This result also follows from \cite{chris}, where it is shown that (\ref{weakest_divergence}) implies $\lambda(W'(\psi))=1$. Since $W'(\psi)\subset W(\psi)$, we have $\lambda(W(\psi))=1$. We use, however, a completely different method of proof.
\end{remark}

\begin{thm}
	\label{app}
	Let $\ve>0$ be given. Assume that  $\psi: \mathbb{N}\to[0,1/2]$ is any function such that (\ref{sum_of_psi}) diverges and either $\psi(m)=0$ or $\psi(n)=0$, whenever two positive integers $m<n$ satisfy $(m,n)\ge n/(\log n)^{\ve}$. Then $\lambda(W(\psi))=1$. 
\end{thm}	
\begin{remark}
	Theorem \ref{app} should be seen in connection with the fact that Khintchine's theorem generally fails without monotonicity. The classical counterexample of Duffin and Schaeffer uses a function $\psi$ which is supported on integers sharing many common prime factors. Theorem~\ref{app}  shows that any counterexample must be of a similar structure: as soon as we can rule out that the integers in the support of $\psi$ have a large common divisor, Khintchine's theorem also holds without monotonicity.
\end{remark}	
\section{Proof of Lemma \ref{S>=neps} and Theorems \ref{khintchine-divergence}, \ref{app}}

The following classical lemma will be crucial for the proof of the theorem. For its proof see, for example,  \cite[Lemma 2.3]{har_b}.
\begin{lemma} 
	\label{limsup}
	Let $\mathcal{A}_n, n=1, 2,\ldots,$ be events in a probability space $(\Omega, \mathcal F, \lambda)$, such that
	$$
	\sum_{n=1}^\infty\lambda(\mathcal{A}_n)=\infty.
	$$
	Then the set $\mathcal{A}$ of points in $\Omega$ belonging to infinitely many sets $\mathcal{A}_n$ satisfies
	$$
	\lambda({\mathcal{A}})\ge\limsup_{N\to\infty}\frac{\left(\sum_{n=1}^N\lambda(\calA_n)\right)^2}{\sum_{1\le m,n\le N}\lambda(\calA_m\cap\calA_n)}.
	$$
\end{lemma}

{\bf Proof of Lemma \ref{S>=neps}.}
 We start with the observation that if $d|n$, the number of $a\in\{1,\dots, n\}$ such that $(a,n)=d$ is exactly $\varphi(n/d)$. Thus,
\begin{align}
\left|S\right|=\sum_{d|n,\  d\le D}\varphi\left(\frac{n}{d}\right)\ge\sum_{d|n,\  d\le D}\frac{\varphi(n)}{d}=\varphi(n)\sum_{d|n,\  d\le D}\frac{1}{d}.
\label{expr S}
\end{align}
Let us denote by $\mathcal{P}$ the set of prime divisors of $n$, which are $\le D$ and let $\mathcal{E}=\{p\le D\}\setminus\mathcal{P}$. Then the last sum in (\ref{expr S}) satisfies the inequality
\begin{align*}
\sum_{d|n,\  d\le D}\frac{1}{d}\:\prod_p\left(1+\frac{1}{p^2}+\frac{1}{p^3}+\cdots\right)\ge\sum_{d\le D\atop p|d \Longrightarrow p\in\mathcal{P}}\frac{1}{d}.
\end{align*}
The product above can be bounded as follows.
\begin{align*}
\prod_p\left(1+\frac{1}{p(p-1)}\right)&=\frac{\zeta(2)\zeta(3)}{\zeta(6)}<2.
\end{align*}
Thus,
\begin{align}
\sum_{d|n,\  d\le D}\frac{1}{d}\ge\frac{1}{2}\sum_{d\le D\atop p|d \Longrightarrow p\in\mathcal{P}}\frac{1}{d}.\label{comp}
\end{align}
We make use of the comparison of the sums in (\ref{comp}), since the sum on the right-hand side can be estimated by the sieve with logarithmic weights. To get a lower bound we follow the lines of the proof of Lemma 2.1 in Granville et al. \cite{gran}. Using their notation, we get
\begin{align*}
\sum_{d\le D\atop p|d \Longrightarrow p\in\mathcal{P}}\frac{1}{d}&=\prod_{p\in\mathcal{E}}\left(1-\frac{1}{p}\right)\prod_{p\in\mathcal{E}}\left(1-\frac{1}{p}\right)^{-1}\sum_{d\le D\atop p|d \Longrightarrow p\in\mathcal{P}}\frac{1}{d}\\
&\ge\prod_{p\in\mathcal{E}}\left(1-\frac{1}{p}\right)\sum_{l\le D\atop p|l \Longrightarrow p\in\mathcal{E}}\frac{1}{l}\sum_{d\le D\atop p|d \Longrightarrow p\in\mathcal{P}}\frac{1}{d}\\
&\ge\prod_{p\in\mathcal{E}}\left(1-\frac{1}{p}\right)\sum_{m\le D}\frac{1}{m}\ge\prod_{p\in\mathcal{E}}\left(1-\frac{1}{p}\right) \log D
\end{align*}
for all $D\ge 1$, with the use of the fact that every integer $m\le D$ may be written as $ld$ in the second to last inequality.
In view of the above, (\ref{expr S}) can be written as
\begin{align}
\left|S\right|\ge \frac{n}{2}\log D\prod_{p\le D}\left(1-\frac{1}{p}\right)\prod_{p|n,\  p>D}\left(1-\frac{1}{p}\right).\label{multiplication}
\end{align}
To evaluate the products appearing in (\ref{multiplication}) we make use of Mertens' formula (see  \cite[p. 17]{ten}). We get
\begin{align*}
\prod_{p\le D}\left(1-\frac{1}{p}\right)=\frac{\e^{-\gamma}+\lo(1)}{\log D}\ge \frac{1}{2\log D},
\end{align*}
for all sufficiently large $D$. We evaluate the second product in (\ref{multiplication}) by splitting it into two parts as follows:
\begin{align*}
\prod_{p|n\atop D<p\le\log n}\left(1-\frac{1}{p}\right)&\ge\frac{\e^{-\gamma}+\lo(1)}{\log\log n}\cdot\frac{\ve\log\log n}{\e^{-\gamma}+\lo(1)}\ge 0.8\ \ve,\\
\prod_{p|n\atop p>\log n}\left(1-\frac{1}{p}\right)&\ge\left(1-\frac{1}{\log n}\right)^{\frac{\log n}{\log\log n}}\ge\exp\left(\frac{\log n}{\log\log n}\cdot\frac{-1}{\log n -1}\right)\ge\frac{1}{2},
\end{align*}
for every sufficiently large $n$, with a use of the estimate $\log(1+x)\ge x/(1+x)$, valid for every $x>-1$, in the second to last inequality.
Inserting the last three calculated bounds in (\ref{multiplication}), for sufficiently large $n$ we obtain the desired inequality.

{\bf Proof of Theorem \ref{khintchine-divergence}.}
 Let $\calE_n:=\mathcal{E}_n^D$, where $D=(\log n)^{\ve/4}$. We prove that the set of $x\in\mathbb{R} / \mathbb{Z}$, which are contained in infinitely many sets $\calE_n$, has positive measure. Then, since $\calE_n\subset\mathcal{E}_n$ we get that $\lambda(W(\psi))>0$. In accordance to Cassels' zero-one law \cite{cassels} it means that $\lambda(W(\psi))=1$, as required.
 
 In view of Lemma~\ref{limsup} it is enough to prove that 
 \begin{align}
 \sum_{1\le m, n\le N}\lambda\left(\calE_m\cap\calE_n\right)\ll\left(\sum_{n\le N}\lambda(\calE_n)\right)^2 \label{overlapO}
 \end{align}
for infinitely many $N$, since we only require "$\limsup$". Assume $n>m\ge 1$ are fixed. Let
$\lambda(\calE_m\cap \calE_n)=B_1+B_2$, where $B_1, B_2$ are the contributions to the intersection from intervals whose centres do not coincide, and the contributions to the intersection from intervals whose centres do coincide (that is, $r/m=s/n$ for some integers $r$ and $s$), respectively. As shown in \cite[p. 39]{har_b}, we have an elementary bound $B_1\le 8\psi(n)\psi(m)$. To investigate $B_2$, suppose we have two overlapping intervals
$$
\left( \frac{r - \psi(m)}{m},\frac{r+\psi(m)}{m} \right)\bigcap \left( \frac{s - \psi(n)}{n},\frac{s+\psi(n)}{n} \right)\neq\emptyset
$$
with $rn=sm$, $(r,m)\le (\log m)^{\ve/4}, (s,n)\le (\log n)^{\ve/4}$. A similar situation is already discussed in \cite[p. 176]{har_b}, but instead of $(\log m)^{\ve/4}, (\log n)^{\ve/4}$ they bound the greatest common divisors by a specific constant. Nonetheless, the following arguments are an adaptation of their proof to our case. We have
$$
B_2\le 2\frac{\psi(n)}{n}\mathop{\sum_{r=1}^{m}\sum_{s=1}^{n}}_{\substack{rn=sm\\ (r,m)\le(\log m)^{\ve/4}\\(s,n)\le(\log n)^{\ve/4}}}1.
$$
It is clear that the summation is empty unless our fixed $m,n$ satisfy $(m,n)\ge n/(\log n)^{\ve/4}$. When solutions to $rn=ms$ exist, there are no more than $m$ of them. Hence for any $N$
\begin{align*}
%\label{sumB2}
\sum_{m=1}^N\sum_{m<n\le N}B_2&\le 2\sum_{n=1}^N\frac{\psi(n)}{n}\sum_{\substack{n>m\ge n/(\log n)^{\ve/4}\\ (n,m)\ge n/(\log n)^{\ve/4}}}m\notag\\
&\le 2\sum_{n=1}^N\frac{\psi(n)}{n}\sum_{\substack{d|n\\ d\le(\log n)^{\ve/4}}}n\,d\notag\\
&\le 2\sum_{n=1}^N\frac{\psi(n)}{n}\cdot n(\log n)^{\ve/2}.\notag
\end{align*}
Combining all the above overlap estimates and adding the summand for $m=n$ we obtain
\begin{align}
\label{sumestimate}
\sum_{m=1}^N\sum_{n=1}^N\lambda\left(\calE_m\cap\calE_n\right)&\le 8\sum_{m=1}^N\sum_{n=1}^N\psi(m)\psi(n)+4\sum_{n=1}^N\psi(n)(\log n)^{\ve/2}+\sum_{n=1}^N \lambda(\calE_n)\notag\\
&\le\left(\frac{8\cdot 20^2}{\ve^2}+1\right) \left(\sum_{n=1}^N\lambda(\calE_n)\right)^2+4\sum_{n=1}^N\psi(n)(\log n)^{\ve/2},
\end{align}
with a use of Corollary~\ref{corollary} in the first summand of (\ref{sumestimate}). For the second summand we apply partial summation to get
\begin{align*}
\sum_{n=1}^N\psi(n)(\log n)^{\ve/2}&=\sum_{n=1}^N\psi(n) (\log N)^{\ve/2}-\psi(1)(\log 2)^{\ve/2}\\
&-\int_{2}^{N}\frac{0.5\,\ve\ \sum_{m\le t}\psi(m)}{t(\log t)^{1-\ve/2}}\ \dd t
\le \sum_{n=1}^N\psi(n) (\log N)^{\ve/2}.
\end{align*}
We are left to prove that $\sum_{n=1}^N\psi(n) (\log N)^{\ve/2}\ll\left(\sum_{n\le N}\lambda(\calE_n)\right)^2$ for infinitely many $N$, which together with (\ref{sumestimate}) will establish the theorem. To show this, we divide the integers $n>4$ into blocks 
$$
2^{2^{k}}<n\le 2^{2^{k+1}}, \qquad k\ge 1.
$$
Then the extra divergence condition (\ref{divergence}) implies that 
$$
\sum_{n=2^{2^{k}}+1}^{2^{2^{k+1}}}\frac{\psi(n)}{(\log n)^\ve}\ge \frac{1}{k^2}
$$
for infinitely many $k$. Thus for such $k$
$$
\sum_{n=1}^{2^{2^{k+1}}}\psi(n)\ge\sum_{n=2^{2^{k}}+1}^{2^{2^{k+1}}}\psi(n)\ge 2^{k \ve}(\log 2)^\ve\frac{1}{k^2},
$$
since in this range of summation $2^{k \ve}(\log 2)^\ve\le(\log n)^\ve\le 2^{(k+1)\ve}(\log 2)^\ve$. Letting $2^{2^{k+1}}=:N$ in the above inequality, we deduce that for infinitely many sufficiently large $N$
$$
\sum_{n=1}^N\psi(n)\ge \left(\frac{\log N}{2}\right)^\ve\cdot\left(\frac{\log 2}{\log\log N-\log\log 2-\log 2}\right)^2\ge\left(\log N\right)^{\ve/2}.
$$

Hence for infinitely many sufficiently large $N$
\begin{align*}
\sum_{n=1}^N\psi(n) (\log N)^{\ve/2}\le\left(\sum_{n=1}^N\psi(n)\right)^2\le\left(\frac{20}{\ve}\right)^2\left(\sum_{n\le N}\lambda(\calE_n)\right)^2,
\end{align*}
which, together with (\ref{sumestimate}), implies that (\ref{overlapO}) holds. This  finishes the proof.
%\ll\left(\sum_{n\le N}\lambda(\calE_n)\right)^2$ for infinitely many $N$,

{\bf Proof of Theorem \ref{app}.}
We follow the first part of the proof of Theorem \ref{khintchine-divergence}. Set $D=(\log n)^\ve$ and consider the sets $\mathcal{E}_n^D$. By our assumption $B_2=0$ and we thus easily get (\ref{overlapO}) with $\calE_n$ replaced by $\mathcal{E}_n^D$. From this the result follows, as noted in the second paragraph of the proof of Theorem \ref{khintchine-divergence}.

\section*{Acknowledgments}
This work was supported by the Austrian Science Fund project F5510-N26. 
This paper was written while I was visiting the number theory group of TU Graz. I would like to thank for a very warm welcoming, good company and the opportunity to be a part of academic life there. Special thanks to Christoph Aistleitner for suggesting the problem and for the helpful discussions.

%%\bibliography{   }
%%\bibliographystyle{abbrv}

\end{document}